\documentclass[a4paper,12pt]{amsart}

\usepackage{amsmath,amsthm,amscd}

\usepackage{amssymb}

\newtheorem{theorem}{Theorem}[section]
\newtheorem{lemma}[theorem]{Lemma}
\newtheorem{corollary}[theorem]{Corollary}
\newtheorem{proposition}[theorem]{Proposition}
\newtheorem{remark}[theorem]{Remark}
\newtheorem{notation}[theorem]{Notation}

\newtheorem{example}[theorem]{Example}

\newcommand{\cJ}{\mathcal{J}}   
  \newcommand{\cO}{\mathcal{O}}

\newcommand{\fa}{\mathfrak{a}} \newcommand{\fb}{\mathfrak{b}} 
  
  \newcommand{\fm}{\mathfrak{m}}
 \newcommand{\fp}{\mathfrak{p}}

 \newcommand{\bbN}{\mathbb{N}}

\newcommand{\ga}{\alpha} \newcommand{\gb}{\beta}

\newcommand{\lrarrow}{\longrightarrow}

\DeclareMathOperator{\lct}{lct}

  \DeclareMathOperator{\Spec}{Spec}

\title{Jumping numbers and ordered tree structures on the dual graph}

\author{Eero Hyry}
\address{Mathematics and Statistics\\
School of Information Sciences\\ 
University of Tampere\\
FIN-33014 Tampereen yliopisto\\ 
Finland}
\email{eero.hyry@uta.fi}
  
\author{Tarmo J\"arvilehto}
\address{P\"a\"askykuja 5, FIN-04620 M\"ants\"al\"a\\Finland}
\email{tarmo.jarvilehto@helsinki.fi}

\date{}

\dedicatory{Dedicated to the memory of Professor Olli Jussila}

\begin{document}

\begin{abstract}
Let $R$ be a two-dimensional regular local ring having an algebraically closed residue field and let $\fa$ 
be a complete ideal of finite colength in $R$.  In this article we investigate the jumping numbers of $\fa$ by means of
the dual graph of the minimal log resolution of the pair $(X,\fa)$. Our main result is a combinatorial criterium for a positive 
rational number $\xi$ to be a jumping number. In particular, we associate to each jumping number 
certain ordered tree structures on the dual graph.
\end{abstract}

\maketitle

\section{Introduction}

Multiplier ideals have in recent years emerged as an important tool in algebraic geometry. Given a closed
subscheme of a smooth complex variety, there is a nested sequence of multiplier ideals parametrized by the positive
rational numbers. A jumping number is a value of the rational parameter at which the multiplier ideal makes a jump.  
Jumping numbers form a discrete set of invariants, which contains important information about the singularities of 
the subscheme in question. 

Jumping numbers are defined by using an embedded resolution of the subcheme. They depend on the exceptional divisors appearing 
in the resolution. The purpose of this article is to look at jumping numbers from the point of view of the combinatorics of 
exceptional divisors in the case of a smooth surface. In particular, we associate to each jumping number certain ordered tree structures 
on the dual graph of the resolution. These structures generalize the one discovered by Veys in~\cite{V} while he was studying poles of the 
topological zeta function 

To describe our work in more detail, let $\fa$ be a complete ideal of finite colength in a two-dimensional regular local ring $R$ having 
an algebraically closed residue field. Let $X\lrarrow \Spec R$ be a minimal log resolution of the pair $(R,\fa)$. Let $D$ be the divisor on $X$ such 
that $\cO_X(-D)=\fa \cO_X$. Let $E_1,\ldots,E_N$ be the exceptional divisors. Recall that a divisor $F=f_1E_1+\ldots +f_NE_N$ is called antinef if  
$F \cdot E_\gamma \le 0$ for all $\gamma=1,\ldots,N$, where $F \cdot E_\gamma$ denotes the intersection product. Our starting point is the observation
made in~\cite{J} that jumping numbers of $\fa$ can be parametrized by the
antinef divisors. More precisely, the jumping number corresponding to $F$ is $$\xi_F:=\min_\gamma \frac{f_\gamma+k_\gamma+1}{d_\gamma},$$ where 
$D=d_1E_1+\ldots+ d_NE_N$ and $K=k_1E_1+\ldots+ k_NE_N$ is the canonical divisor.

Let $\Gamma$ denote the dual graph of $X$. Recall that the vertices of the dual graph correspond to the exceptional divisors and that two vertices are 
adjacent if and only if the corresponding exceptional divisors intersect. The above considerations motivate us to investigate the function $\lambda(f_\gamma,\gamma)$ on $|\Gamma|$,  where $$\lambda(a,\gamma):=\frac{a+k_\gamma+1}{d_\gamma}$$ for any $a\in \mathbb Z$ and $\gamma \in |\Gamma|$. By the above the minimum value of $\lambda(f_\gamma,\gamma)$ is now the jumping number $\xi=\xi_F$. We call the set $$S_F:=\{\gamma \in |\Gamma| \mid \lambda(f_\gamma,\gamma)=\xi\}$$ the support of $\xi$ with respect to the antinef divisor $F$. 

Our main result Theorem~\ref{hyppylukukriteeri} is a combinatorial criterium for a positive rational number $\xi$ to be a jumping number. We observe that it is possible to choose the divisor $F$ in such a way that $\lambda(f_\gamma,\gamma)$ strictly increases along every path away from the support. 
By assigning the number $\lambda(f_\gamma,\gamma)$ to each vertex $\gamma$, we can  make the dual graph an ordered tree.
An end of $S_F$ must either have at least three adjacent vertices or correspond to some Rees valuation of the ideal. Moreover, one may assume that $S_F$ is a chain such that the vertices corresponding to Rees valuations do not occur at the non-ends of $S_F$.  

We also want to understand how the so called contributing divisors arise. The notion of contribution to a jumping number by a divisor 
was introduced by Smith and Thompson in~\cite{ST}, and the investigation has then been continued by Tucker in~\cite{T}. It turns out in Theorem~\ref{kriittinenlause} that every critically contributing divisor, in the sense of Tucker, is of the type $\sum_{\gamma\in S_F}E_\gamma$. However, the converse is not true. Therefore we give in Theorem~\ref{kriittinenkriteeri} a necessary and sufficient condition for a reduced divisor to be a critically contributing one. 

Using Theorem~\ref{hyppylukukriteeri} we will also show in Corollary~\ref{hyppyluvut} that given a vertex with at least three adjacent vertices or a vertex corresponding to a Rees valuation of the ideal, there is always a jumping number supported exactly at this vertex. Note that a support of a jumping number always contains vertices of this type. Moreover, by means of Theorem~\ref{hyppylukukriteeri} we can in Proposition~\ref{uudet} construct from a given jumping number certain new jumping numbers having the same support as the original one.

Our main technical tool is Lemma~\ref{veys} which helps us to construct suitable antinef divisors $F$. This is inspired by the work of Loeser and Veys concerning the numerical
data associated to the exceptional divisors of the resolution (see~\cite{Lo} and~\cite{V}). 

Finally, we would like to refer to the book of Favre and Jonsson (\cite{FJ}) for related topics. It would be interesting to know whether our results can be interpreted in their `tree language'.

\section{Preliminaries}

We begin by fixing notation and recalling some basic facts from the Zariski-Lipman theory of complete ideals. For more details, we refer to~\cite{L},
\cite{C}, \cite{LW} and~\cite{J}. 

Throughout this article $R$ denotes a regular local ring of dimension two having an algebraically closed residue field. Let
$\fa$ be a complete ideal of finite colength in $R$. Let $\pi \colon X\rightarrow\Spec(R)$ be a minimal principalization of $\fa$.
Then $X$ is a regular scheme and $\fa\cO_X=\cO_X(-D)$ for an effective Cartier divisor $D$. Note that $\pi$ is a log resolution
of the pair $(\Spec R,\fa)$, i.e., the divisor $D+\textrm{Exc}(\pi)$ has simple normal crossing support, where $\textrm{Exc}(\pi)$ 
denotes the sum of the exceptional divisors of $\pi$.

The morphism $\pi$ is a composition of point blowups of regular schemes 
\begin{equation*}
\pi:X=X_{N+1}\xrightarrow{\pi_N}\cdots\xrightarrow{\pi_2} X_2\xrightarrow{\pi_1}
X_1=\Spec R,
\end{equation*} 
where $\pi_\mu: X_{\mu+1}\rightarrow X_\mu$ is the blowup of 
$X_\mu$ at a closed point $x_\mu\in X_\mu$ for every $\mu=1,\dots,N$. 
Let $E_\mu$ and $E^*_\mu$, respectively, be the strict and total transform of the exceptional divisor $\pi^{-1}_\mu\{x_\mu\}$ 
on $X$ for every $\mu=1,\dots,N$. We denote by $v_\mu$ the discrete valuation associated to the discrete valuation ring $\cO_{X,E_\mu}$, in other words,
$v_\mu$ is the $\fm_{X_\mu,x_\mu}$-adic order valuation.

Recall that a point $x_\mu$ is said to be infinitely near to a point $x_\nu$, if the projection $X_\mu \rightarrow X_\nu$ maps $x_\mu$
to $x_\nu$. This relation gives a partial order on the set $\{x_1,\ldots,x_N\}$. A point $x_\mu$ is \textit{proximate} to the point 
$x_\nu$, denoted by $\mu\succ \nu$, if and only if $x_\mu$ lies on the strict 
transform of $\pi_\nu^{-1}\{x_\nu\}$ on $X_\mu.$ Following \cite[Definition-Lemma 1.5]{C}, the \textit{proximity matrix} is
\begin{equation*}\label{prox}
P:=(p_{\mu,\nu})_{N\times N},\text{ where }p_{\mu,\nu}=
\left\lbrace
\begin{array}{rl}
   1,&\text{if }\mu=\nu;\\
  -1,&\text{if }\mu\succ \nu;\\
   0,&\text{otherwise.}
\end{array}
\right.
\end{equation*}
Note that this is the transpose of the proximity matrix given in \cite[p. 6]{L}. We set $Q=(q_{\mu,\nu})_{N\times N}:=P^{-1}$. 
The equation $PQ=1$ immediately gives the formula
\begin{equation}\label{Q}q_{\mu,\nu}=\sum_{\mu\succ \rho}q_{\rho,\nu}+\delta_{\mu,\nu}.\end{equation}
If $x_\mu$ is infinitely near to $x_\nu$, then $q_{\mu,\nu}>0$ while $q_{\mu,\nu}=0$ otherwise. Clearly $q_{\mu,\mu}=1$ for all $\mu=1,\ldots,N$. 

We denote by $\Gamma$ the \textit{dual graph} associated to our principalization. It is well known that $\Gamma$ is a tree. Let $|\Gamma|$ be the corresponding set of vertices. Recall that there is a vertex $\nu$ corresponding to each exceptional divisor $E_\nu$ weighted by the number $w_\Gamma(\nu):=-E_\nu^2$. Note that
$$w_\Gamma(\nu)=1+\#\{\mu\in |\Gamma|\mid \mu \succ \nu\}.$$ Two vertices $\mu$ and $\nu$ are called \textit{adjacent} if they can be joined by an edge. This is the case if and only if the corresponding exceptional divisors  $E_\mu$ and $E_\nu$ intersect. We write $\mu\sim\nu$. Then either $\mu \prec \nu$ or $\mu \succ \nu$. Suppose, for example, that $\mu \succ  \nu$. Then 
$\nu\sim\mu$ in fact means that $x_\mu$ is a maximal element in the set of infinitely near points proximate to $x_\nu$. 
The \textit{valence} $v_\Gamma(\nu)$ of a vertex $\nu$ means the number of vertices adjacent to it. 
If $v_\Gamma(\nu)\ge 2$, then $\nu$ is called a \textit{star}. A vertex $\tau$ with $v_\Gamma(\tau)=1$ is an \textit{end}.
The \textit{distance} between two vertices $\mu, \nu\in |\Gamma|$ is defined as
$$
d(\nu,\mu):=\min\{r\mid\nu=\nu_0\sim\cdots\sim\nu_r=\mu,\text{ where }
\nu_0,\dots,\nu_r\in|\Gamma|\}.
$$
Furthermore, if $S\subset |\Gamma|$, we set 
$$d(\nu,S):=\min\{d(\nu,\mu)\mid\mu\in S\}.$$
If $d(\nu,S)=1$, then we write $\nu\sim S$.

We consider the lattice $\Lambda:= \mathbb Z E_1 \oplus  \ldots \oplus \mathbb Z E_N$ of exceptional divisors on $X$.  
The lattice $\Lambda$ 
has two other convenient bases 
besides $\{E_\mu\mid \mu=1,\dots,N\}$, namely $\{E_\mu^*\mid \mu=1,\dots,N\}$ and $\{\widehat E_\mu\mid \mu=1,\dots,N\}$, where  
$E_\mu\cdot\widehat E_\nu=-\delta_{\mu,\nu}$ for $\mu,\nu=1,\dots,N$. For any $G\in\Lambda$, we write
$$G=g_1E_1+\ldots+g_NE_N=g_1^*E^*_1+\ldots+g_N^*E^*_N=\widehat g_1\widehat E_1+\ldots+\widehat g_N\widehat E_N.$$
The following base change formulas now hold: \begin{equation}\label{BC}
g^*=gP^\textsc{t}\text{ and }\widehat g=gP^\textsc{t}P=g^*P,  
\end{equation}
where $g,g^*$ and $\widehat g$ denote row vectors in $\mathbb Z^n$. Here $(P^\textsc{t}P)_{\mu,\nu}=-
E_\mu\cdot E_\nu$. In particular, note the formulas
\begin{equation}\label{P}
\widehat g_\mu=g^*_\mu - \sum_{\nu\succ \mu}g^*_\nu=w_\Gamma(\gamma)g_\mu - \sum_{\nu \sim \mu}g_\nu\quad (\mu=1,\ldots,N).
\end{equation}
The \textit{support} of a divisor $G\in\Lambda$ is $|G|:=\{\gamma\in |\Gamma|\mid g_\gamma\not=0\}$.

Recall that a divisor $F\in \Lambda$ is \textit{antinef} if $\widehat f_\nu=-F \cdot E_\nu \ge 0$ for all $\nu=1,\ldots,N$. Equivalently, the \textit{proximity inequalities}
\begin{equation}\label{Ptahti}f^*_\mu \ge \sum_{\nu\succ \mu}f^*_\nu\quad (\mu=1,\ldots,N)\end{equation}
hold. Note that they can also be expressed in the form
\begin{equation}\label{Pvalu}w_\Gamma(\mu)f_\mu \ge \sum_{\nu \sim \mu}f_\nu\quad (\mu=1,\ldots,N).\end{equation}
In fact, if $F\not=0$ is antinef, then also $f_\nu>0$ for all $\nu=1,\ldots,N$. There is a one to one correspondence between the antinef divisors in $\Lambda$ and the complete ideals of finite colength in $R$ generating invertible $\cO_X$-sheaves, given by $F\leftrightarrow\Gamma(X,\cO_X(-F))$. For a divisor $G\in\Lambda$, there exists a minimal one among the antinef divisors $F$ satisfying $F\ge G$. This is called the \textit{antinef closure} of $G$ and denoted by $G^\sim$. We have \begin{equation*}\Gamma(X,\cO_X(-G))=\Gamma(X,\cO_X(-G^\sim))
\end{equation*} for any divisor $G\in\Lambda$.

Recall that an ideal is called \textit{simple} if it cannot be expressed as a product of two proper ideals. By the famous result of Zariski,
every complete ideal factorizes uniquely into a product of simple complete ideals. More precisely, $$\fa=\fp_1^{\widehat d_1}\cdots\fp_N^{\widehat d_N},$$ where $\fp_\mu\subset R$ denotes the simple complete ideal of finite colength corresponding to the exceptional divisor $E_\mu$ and $\widehat d_\mu>0$ if and only if $v_\mu$ is a Rees valuation of $\fa$. We have $\fp_\mu\cO_X=\cO_X(-\widehat E_\mu)$ so that $\fp_\mu=\Gamma(X,\cO_X(-\widehat E_\mu))$. 
By~(\ref{BC})
\begin{equation}\label{E}\widehat E_\mu=\sum_\nu q_{\mu,\nu}E^*_\nu=\sum_{\nu,\rho}q_{\nu,\rho}q_{\mu,\rho}E_\nu.\end{equation} 
In particular, we observe the \textit{reciprocity formula}
\begin{equation}\label{reciprocity}v_\nu(\fp_\mu)=\sum_{\rho} q_{\nu,\rho}q_{\mu,\rho}=v_\mu(\fp_\nu)\quad (\mu,\nu=1,\ldots,N).\end{equation} 
For $\mu\not=\nu$, the proximity inequalities now become equalities $$q_{\mu,\nu} = \sum_{\rho \succ \nu}q_{\mu,\rho}.$$

We will next recall the definition of jumping numbers. A general reference for jumping numbers
is the fundamental article~\cite{ELSV}. Recall first that the \textit{canonical divisor} is $K=\sum_\nu E^*_\nu$. The formulas (\ref{BC}) now give 
\begin{equation}\label{K}k_\nu=\sum_\mu q_{\nu,\mu}\quad\hbox{and}\quad \widehat k_\nu=E_\nu^2+2\quad (\nu=1,\ldots,N).\end{equation} 
By (\ref{P}) 
we then obtain an important relation 
\begin{equation}\label{KP}
w_\Gamma(\nu)k_\nu=2-w_\Gamma(\nu)+\sum_{\mu\sim \gamma}k_\mu\quad(\nu=1,\ldots,N).
\end{equation}

For a nonnegative rational number $\xi$, the \textit{multiplier ideal} $\cJ(\fa^\xi)$ is defined to be the ideal 
\[
\cJ(\fa^\xi):=\Gamma\left(X,\cO_X\left(K-\left\lfloor \xi D\right\rfloor\right)\right)\subset R,
\]
where $\left\lfloor \xi D\right\rfloor$ denotes the integer part of $\xi D$. 
It is now known that there is an increasing discrete sequence 
$$0=\xi_0<\xi_1<\xi_2<\cdots$$ of rational numbers 
$\xi_i$ characterized by the properties that $\cJ(\fa^\xi)=\cJ(\fa^{\xi_i})$ for $\xi\in[\xi_i,\xi_{i+1})$, while 
$\cJ(\fa^{\xi_{i+1}})\subsetneq\cJ(\fa^{\xi_i})$ for every $i$. The numbers $\xi_1,\xi_2,\dots$, are called the \textit{jumping 
numbers} of $\fa$. Note that contrary to \cite[Definition 1.4]{ELSV}, we don't consider $0$ as a jumping number. Clearly, 
this is no restriction. The following Proposition~\ref{parametrisointi}, which is fundamental for the rest of this article, 
results from~\cite[Proposition 6.7 and Proposition 7.2]{J}.

\begin{proposition}\label{parametrisointi}
Let $R$ be a two-dimensional regular local ring and let $\fa \subset R$ be a complete ideal of finite 
colength. Then $\xi$ is a jumping number of $\fa$ if and only if there exists an antinef divisor $F\in \Lambda$
such that $$\xi=\xi_F:=\min_\nu \frac{f_\nu+k_\nu+1}{d_\nu}.$$ Moreover, if $\fb$ is the complete ideal corresponding to
$F$, then $$\xi=\inf\{c\in\mathbb Q_{>0}\mid\cJ(\fa^{c})\nsupseteq\fb\}.$$
\end{proposition}

\begin{notation}
We write for any integer $a$ and for any vertex $\nu$ 
$$\lambda(a,\nu):=\frac{a+k_\nu+1}{d_\nu}$$ and call
the set $$S_F:=\{\nu \in |\Gamma|\mid \lambda(f_\nu,\nu)=\xi\}$$ 
the \textit{support} of the jumping number $\xi$ with respect to the
antinef divisor $F$.
\end{notation}

\section{Relations between numerical data associated to exceptional divisors}

Let $\xi$ be a jumping number of the ideal $\fa$. In order to define an ordered tree structure on the dual graph as described in the introduction, 
we must be able to construct an antinef divisor $F$ such that $\xi=\xi_F$ and that $\lambda(f_\gamma,\gamma)$ increases along every path away from the support $S_F$. Proceeding inductively, suppose that we are given a vertex $\gamma$ and a vertex $\eta_1\sim \gamma$ such that $d(\gamma,S_F)>d(\eta_1,S_F)$. Suppose, furthermore, that numbers $f_\gamma$ and $f_{\eta_1}$ have been defined in such a way that $\lambda(f_\gamma,\gamma)>\lambda(f_{\eta_1},\eta_1)$. 
The key issue is to find for vertices $\eta_1\not=\eta \sim \gamma$ suitable numbers $f_\eta$ with the property that $\lambda(f_\eta,\eta)
>\lambda(f_\gamma,\gamma)$. This problem will be addressed in Lemma~\ref{veyslemma}, which is the main result of this section. Other details of
the above construction will be postponed till Lemma~\ref{vieritys} in the next section. 

We begin with the following lemma:

\begin{lemma}\label{ketju} Let $F\in\Lambda$ be a divisor. Let $\gamma\in|\Gamma|$ be a vertex such that $\widehat 
f_\gamma\ge0$ and $\xi:=\lambda(f_\gamma,\gamma)\le\lambda(f_\eta,\eta)$ when $\eta\sim\gamma$.
Then 
$$\xi \widehat d_\gamma \ge {\widehat f}_\gamma-v_\Gamma(\gamma)+2 .$$
In particular, this implies the following:
\begin{itemize}
\item[a)]
If $v_\Gamma(\gamma)\le 2$ and $\widehat d_\gamma=0$, then there are exactly two vertices $\eta$ adjacent to $\gamma$ and
$\lambda(f_{\eta}, \eta)=\xi$ for both of those. Furthermore $\widehat f_\gamma=0$.

\item[b)]If $v_\Gamma(\gamma)=1$, then $\xi \widehat d_\gamma \ge \widehat f_\gamma+1$. Especially, $\widehat d_\gamma>0$.
\end{itemize}

\end{lemma}

\begin{proof} 

By the formulas $(\ref{P})$ and $(\ref{KP})$ we obtain
$$
\xi=\frac{w_\Gamma(\gamma)(f_\gamma+k_\gamma+1)}
{w_\Gamma(\gamma)d_\gamma}
=\frac{\sum_{\eta\sim\gamma}(f_\eta+k_\eta+1)+{\widehat f}_\gamma-v_\Gamma(\gamma)+2}
{\sum_{\eta\sim\gamma}d_\eta+{\widehat d}_\gamma}.
$$
Since
$$
\xi\le\frac{f_\eta+k_\eta+1}{d_\eta},
$$
we must have 
$ \xi \widehat d_\gamma \ge {\widehat f}_\gamma-v_\Gamma(\gamma)+2$.

When $v_\Gamma(\gamma) \le 1$, this immediately gives $\widehat d_\gamma>0$. Then suppose that
$\widehat d_\gamma=0$ and $v_\Gamma(\gamma)= 2$. Now ${\widehat f}_\gamma-v_\Gamma(\gamma)+2\le 0$ 
is possible only if $$\widehat f_\gamma=v_\Gamma(\gamma)-2=0.$$ But then
$$
\xi
=\frac{\sum_{\eta\sim\gamma}(f_\eta+k_\eta+1)}
{\sum_{\eta\sim\gamma}d_\eta},
$$
which implies that $\xi=\lambda(f_{\eta},\eta)$ for both $\eta\sim\gamma$.
\end{proof}

For any two vertices $\nu,\gamma\in|\Gamma|$, set
$$\ga_{\gamma,\nu}:=k_{\nu}+1-\frac{k_\gamma+1}{d_{\gamma}}d_{\nu}
=d_\nu(\lambda(0,\nu)-\lambda(0,\gamma)).
$$
The numbers $\ga_{\gamma,\nu}$ were first investigated by Loeser in~\cite{Lo} in the case of an embedded resolution
of a curve. Van Proeyen and Veys generalized his results to the ideal case in~\cite{PV}. The following 
Lemma~\ref{veysequation} and Lemma~\ref{proyenveys} are due to them (\cite[Proposition 3.1 and Corollary 3.2]{PV}).
Since we have slightly modified the statements, we include the proofs here for the convenience of the reader. 
Moreover, working in a more algebraic context, we also prefer to prove these results directly without utilizing the 
results of Loeser.

\begin{lemma}\label{veysequation} If $\gamma\in |\Gamma|$, then 
$$\sum_{\nu\sim\gamma}\ga_{\gamma,\nu}=v_\Gamma(\gamma)-2+\widehat d_\gamma\frac{k_\gamma+1}{d_{\gamma}}.$$
\end{lemma}

\begin{proof}
By the relation $(\ref{KP})$ 
\begin{equation*}\label{gaga1}
\sum_{\nu\sim\gamma}k_{\nu}=w_\Gamma(\gamma)k_\gamma+w_\Gamma(\gamma)-2
\end{equation*}
and by the formula $(\ref{P})$
\begin{equation*}\label{gaga2}
\sum_{\nu\sim\gamma}d_{\nu}=w_\Gamma(\gamma)d_\gamma-\widehat d_\gamma.
\end{equation*}
It follows that
\begin{eqnarray*}
\sum_{\nu\sim\gamma}\ga_{\gamma,\nu}
&=&\displaystyle \sum_{\nu\sim\gamma}k_{\nu}+v_\Gamma(\gamma)-
\sum_{\nu\sim\gamma}d_{\nu}\frac{k_\gamma+1}{d_{\gamma}}\\
&=&\displaystyle w_\Gamma(\gamma)\left(k_\gamma+1\right)+
v_\Gamma(\gamma)-2
-\left(w_\Gamma(\gamma)d_\gamma-\widehat d_\gamma\right)\frac{k_\gamma+1}{d_{\gamma}}\phantom{t}\\
&=&\displaystyle v_\Gamma(\gamma)-2+\widehat d_\gamma\frac{k_\gamma+1}{d_{\gamma}}.
\end{eqnarray*}
\end{proof}

\begin{lemma}\label{proyenveys}Let $\gamma\in |\Gamma|$.
\begin{itemize}
\item[a)] For every $\eta \sim\gamma$, either $$|\ga_{\gamma,\eta}|<1,$$
or $\widehat d_\gamma=0$ and $\eta$ is the only vertex adjacent to $\gamma$, in which case $\ga_{\gamma,\eta}=-1.$ 
\item[b)]
For all except at most one $\eta\sim\gamma$,
$\ga_{\gamma,\eta}\ge0$. More precisely, if $\ga_{\gamma,\eta}\le 0$ for some $\eta\sim\gamma$, then $\ga_{\gamma,\nu}>0$ for all $\eta\neq\nu\sim\gamma$, unless $\ga_{\gamma,\eta}= 0$, 
$\widehat d_\gamma=0$ and
there are only two vertices $\eta'\sim\gamma\sim\eta$, in which 
case also $\ga_{\gamma,\eta'}=0$.
\end{itemize}
\end{lemma}

\begin{proof} 

\smallskip
\noindent
a) Consider the sequence of point blowups
\begin{equation}\label{original}
\pi:X=X_{N+1}\xrightarrow{\pi_N}\cdots\xrightarrow{\pi_2} X_2\xrightarrow{\pi_1}
X_1=\Spec R,
\end{equation} 
where $\pi_\mu: X_{\mu+1}\rightarrow X_\mu$ is the blowup of 
$X_\mu$ at the closed point $x_\mu\in X_\mu$ for every $\mu=1,\dots,N$. 
We shall proceed by induction on $N$, the case $N=1$ being trivial. Consider the exceptional divisor $E_N$ arising in the last blowup. 
Depending on whether $x_N$ lies in an intersection of two exceptional divisor or not, $E_N$ intersects one or two exceptional divisors.  
In other words, the vertex $N$ is adjacent to one or two vertices.

Suppose first that the vertex $N$ is adjacent to only one vertex $\gb$. We consider the ideal 
$$\fa':=\fp_{\gb}^{\widehat d_N}\prod_{\nu\not=N}\fp_\nu^{\widehat d_\nu}.$$ It has the minimal principalization
\begin{equation}
X_N\xrightarrow{\pi_{N-1}}\cdots\xrightarrow{\pi_2} X_2\xrightarrow{\pi_1}
X_1=\Spec R
\end{equation} 
The associated proximity matrix and its inverse are clearly restrictions of those of (\ref{original}).
So $k'_\nu=k_\nu$ for every $\nu\not=N$. Because $N$ is adjacent only to $\gb$ and $E_N^2=-1$, the proximity equation  
for the ideal $\fp_\nu$ gives $v_N(\fp_\nu)=v_\gb(\fp_\nu)$. By the reciprocity formula (\ref{reciprocity}) we then obtain
$$v_\nu(\fp_N)=v_N(\fp_\nu)=v_\gb(\fp_\nu)=v_\nu(\fp_\gb).$$ Hence $$d'_\nu=\widehat d_Nv_\nu(\fp_\gb) + \sum_{\mu\not=N}\widehat d_\mu 
v_\nu(\fp_\mu) =\widehat d_N v_\nu(\fp_N) + \sum_{\mu\not=N}\widehat d_\mu v_\nu(\fp_\mu)=d_\nu$$ for all $\nu\not=N$. 
By the induction hypothesis, the claim therefore holds if $\gamma,\eta \neq N$. It remains to consider the numbers $\ga_{\gb,N}$ and $\ga_{N,\gb}$.

We first observe that $N$ being proximate only to $\gb$, we have $q_{N,\nu}=q_{\gb,\nu}$ for all $\nu\not=N$ by (\ref{Q}). By~(\ref{K}) we thus get
$$k_N=\sum_\nu  q_{N,\nu}=\sum_{\nu\not=N} q_{\gb, \nu}+1=k_\gb+1.$$ Moreover, the base change formula (\ref{BC}) gives 
$$\widehat d_N =\sum_{\nu}d_\nu (-E_\nu \cdot E_N)=d_{N}- d_\gb.$$ 
We then get
$$
\ga_{\gb,N}
=k_\gb+2-\frac{k_\gb+1}{d_\gb}(d_\gb+\widehat d_N)
=1-\widehat d_N \frac{k_\gb+1}{d_\gb}=1-\widehat d_N \frac{k_\gb+1}{d'_\gb}.
$$
Using (\ref{reciprocity}), we have 
$$v_\gb(\fp_\gb)=\sum_\nu q_{\gb,\nu}^2\ge \sum_\nu q_{\gb,\nu}=k_\gb.$$
Thus  
$$d'_\gb=\widehat d_N v_\gb(\fp_\gb)+\sum_{\nu \not=N} \widehat d_\nu v_\gb(\fp_\nu)\ge \widehat d_N k_\gb+\sum_{\nu \not=N} \widehat d_\nu v_\gb(\fp_\nu)$$ so that
$$
0<\widehat d_N \frac{k_\gb+1}{d'_\gb} \le \frac{k_\gb+1}{k_\gb+\sum_{\nu \not=N}(\widehat d_\nu/\widehat d_N) v_\gb(\fp_\nu)}\le 
\frac{k_\gb+1}{k_\gb}\le 2.
$$
Therefore $-1\le \ga_{\gb,N}< 1$. Moreover, by the above the equality $\ga_{\gb,N}=-1$ can take place only if 
$k_\gb=1$ and $\widehat d_\nu=0$ for all $\nu\not=N$. This means that we must have $\gb=1$ and $\fa=\fp_N^{\widehat d_N}$. In 
particular, $N$ is the only vertex adjacent to $\gb$ and $\widehat d_\gb=0$. Thus a) holds for $\ga_{\gb,N}$. 
In order to show that a) holds for $\ga_{N,\gb}$, too, we note that  
$$
\ga_{N,\gb}=-\ga_{\gb,N}\frac{d_\gb}{d_N}=-\ga_{\gb,N} \frac{d_\gb}{d_\gb+\widehat d_N}.
$$
As necessarily $\widehat d_N>0$, we see that $|\ga_{N,\gb}|<1$.

Let us then consider the case where the vertex $N$ is adjacent to two vertices, say $\gb$ and $\gb'$. 
Now let
$$
\fa':=\fp_{\gb}^{\widehat d_N}\fp_{\gb'}^{\widehat d_N}\prod_{\nu\not=N}\fp_\nu^{\widehat d_\nu}.
$$
The ideal $\fa'$ has a similar minimal principalization as in the preceeding case.  Analogously one obtains $d'_\nu=d_\nu$,
$k'_\nu=k_\nu$ for $\nu\not=N$, $d_N=d_{\gb}+d_{\gb'}+\widehat d_N$ and $k_N=k_{\gb}+k_{\gb'}+1$. 
By induction we only need to prove the claim for $\ga_{N,\gb}$, $\ga_{N,\gb'}$ and $\ga_{\gb,N}$, 
$\ga_{\gb',N}$. We consider here only $\ga_{N,\gb}$ and $\ga_{\gb,N}$, the proof for 
$\ga_{N,\gb'}$ and $\ga_{\gb',N}$ being similar.

Set $m:=v_\Gamma(\gb)$. Note that $\gb'\sim_{\Gamma'}\gb$, where $\Gamma'$ denotes the dual graph of the minimal principalization
of $\fa'$. Therefore we have $m=v_{\Gamma'}(\gb)$, too. By Lemma~\ref{veysequation} we get $$\sum_{\nu\sim_\Gamma\gb}\ga_{\gb,\nu}=
m-2+\widehat d_\gb\frac{k_\gb+1}{d_\gb}$$
and 
$$
\sum_{\nu\sim_{\Gamma'}\gb}\ga_{\gb,\nu}=m-2+(\widehat d_\gb+\widehat d_N)\frac{k_\gb+1}{d_\gb}.
$$
Together the above equations imply that
$$
\ga_{\gb,N}=\ga_{\gb,\gb'}-\widehat d_N\frac{k_\gb+1}{d_\gb}.
$$
By the induction hypothesis, $\ga_{\gb,\gb'}<1$. Thereby also $\ga_{\gb,N} <1$. 
We still need to show that $\ga_{\gb,N}\ge -1$. As $a_{\gb,\nu}<1$ for every $N\neq\nu\sim\gb$, we see that 
$$
\sum_{N\neq\nu\sim_\Gamma\gb}\ga_{\gb,\nu}\le m-1,
$$
where the equality holds if and only if $m=1$.
Hence
$$
\ga_{\gb,N}=\sum_{\nu\sim_\Gamma\gb}\ga_{\gb,\nu}-\sum_{N\neq\nu\sim_\Gamma\gb}\ga_{\gb,\nu}\ge-1+\widehat d_\gb\frac{k_\gb+1}{d_\gb}
\ge-1.
$$
We conclude that $|\ga_{\gb,N}|<1$ unless $m=1$ and $\widehat d_\gb=0$, in which case $\ga_{\gb,N}=-1$.
Finally, we have
$$
\ga_{N,\gb}=-\ga_{\gb, N}\frac{d_\gb}{d_N}=-\ga_{\gb,N}\frac{d_\gb}{d_{\gb}+d_{\gb'}+\widehat d_N}.
$$
Because $|\ga_{\gb, N}|\le 1$ and $\widehat d_N>0$, we obtain $|\ga_{N,\gb}|<1$ as wanted.

\medskip

\noindent b) Obviously, there is nothing to prove if $\gamma$ has only one adjacent vertex. Thus we may assume 
that $m:=v_\Gamma(\gamma)\ge 2$. Suppose that there are two vertices adjacent to $\gamma$, say $\eta$ and 
$\eta'$, with $\ga_{\gamma,\eta}, \ga_{\gamma,\eta'} \le 0$. It then follows from Lemma~\ref{veysequation} that
$$
\sum_{\eta,\eta'\neq\nu\sim\gamma}\ga_{\gamma,\nu}
\ge m-2.
$$
On the other hand, $\ga_{\gamma,\nu}<1$ for every $\nu\sim\gamma$ by a). Subsequently, if $m>2$, then
$$
m-2 >\sum_{\eta,\eta'\neq\nu\sim\gamma}\ga_{\gamma,\nu}. 
$$
Hence necessarily $m=2$, and further, by Lemma~\ref{veysequation} we also have $\ga_{\gamma,\eta}=0=\ga_{\gamma,\eta'}$ and 
$\widehat d_\gamma=0$.
\end{proof}

\begin{notation}
Let $\gamma \in |\Gamma|$ and let $a_\gamma$ be a nonnegative integer. Suppose that $$\{\eta\mid\eta\sim\gamma\}=
\{\eta_1,\dots,\eta_m\}.$$ Set $$\Delta_j:=d_{\eta_j}\lambda(a_\gamma,\gamma)-k_{\eta_j}-1\quad (j=1,\ldots,m)$$
so that
$$\lambda(\Delta_j,\eta_j)=\lambda(a_\gamma,\gamma).$$
Also write 
$$\Delta_j=\left\lfloor\Delta_j\right\rfloor+\delta_j$$ where $0\le\delta_j<1$.
\end{notation}

\begin{lemma}\label{veyslemma} With the preceeding notation, we have
\begin{equation*}
\sum_{j=1}^m\Delta_j+\widehat d_\gamma\lambda(a_\gamma,\gamma)+m-2=w_\Gamma(\gamma)a_\gamma.
\end{equation*}
In particular,
$$
\sum_{j=1}^m\left\lfloor\Delta_j\right\rfloor+m-2=w_\Gamma(\gamma)a_\gamma-\varphi,
$$
where 
$$\varphi:=\sum_{j=1}^m\delta_j+\widehat d_\gamma\lambda(a_\gamma,\gamma)$$ is a nonnegative integer. 
\end{lemma}
	
\begin{proof}
By the formulas $(\ref{P})$ and $(\ref{KP})$ 
$$
\lambda(a_\gamma,\gamma)=\frac{w_\Gamma(\gamma)(a_\gamma+k_\gamma+1)}{w_\Gamma(\gamma)d_\gamma}
=\frac{w_\Gamma(\gamma)a_\gamma+\sum_{j=1}^mk_{\eta_j}+2}
{\sum_{j=1}^m d_{\eta_j}+\widehat d_\gamma}.$$
Therefore
\begin{equation*}
\sum_{j=1}^m\Delta_j+\widehat d_\gamma\lambda(a_\gamma,\gamma)+m-2=w_\Gamma(\gamma)a_\gamma.
\end{equation*}
The last statement is now obvious.
\end{proof}

The following lemma, which can be considered as a generalization of Lemma~\ref{proyenveys}, will play a crucial role in the sequel: 

\begin{lemma}\label{veys}
Given any vertex $\gamma\in |\Gamma|$ and any nonnegative integer $a_\gamma$, we may choose for
every vertex $\eta\sim\gamma$ a nonnegative integer $a_\eta$ so that 
$$
w_\Gamma(\gamma)a_\gamma=\sum_{\eta\sim\gamma}a_\eta\hspace{4pt}
\text{ and }\hspace{4pt}\lambda(a_\eta,\eta)\ge\lambda(a_\gamma,\gamma),
$$ 
where the latter inequality holds for each $\eta\sim\gamma$ except at most one. More precisely, if $$\{\eta\mid\eta\sim\gamma\}=\{\eta_1,\dots,\eta_m\},$$
where $m>1$, then the following is true: 

\begin{itemize}

\item[1)] If it is possible to find a nonnegative integer $a_{\eta_1}$ with $\lambda(a_{\eta_1},\eta_1)=\lambda(a_\gamma,\gamma)$, then one may choose the other integers $a_{\eta_j}$ so that  
$$\lambda(a_{\eta_j},\eta_j)\ge \lambda(a_\gamma,\gamma)$$
for all $1<j\le m$. Moreover, this choice can be done in such a way that the strict inequality holds for all except at most one $1<j\le m$. 
In the case we already have a nonnegative integer $a_{\eta_2}$ with $\lambda(a_{\eta_2},\eta_2) = \lambda(a_\gamma,\gamma)$, we can assume
that the inequality is strict for all $2 < j \le m$.

\item[2)] If it is possible to find a nonnegative integer $a_{\eta_1}$ with $\lambda(a_{\eta_1},\eta_1)<\lambda(a_\gamma,\gamma)$ or, in the case $\widehat d_\gamma>0$, $\lambda(a_{\eta_1},\eta_1)=\lambda(a_\gamma,\gamma)$, 
then one can choose the other integers $a_\eta$ in such a way that 
$$\lambda(a_{\eta_j},\eta_j)> \lambda(a_\gamma,\gamma)$$
holds for every $1<j\le m$.

\end{itemize}

\end{lemma}
	
\begin{proof}
Obviously we may assume that $m>1$. 

In the case $a_\gamma=0$, we set $a_\eta=0$ for every $\eta\sim \gamma$.
Since $$\lambda(a_\eta,\eta)-\lambda(a_\gamma,\gamma)=\frac{\ga_{\gamma,\eta}}{d_\eta},$$
the claim follows from Lemma~\ref{proyenveys}. 

Suppose thus that $a_\gamma>0$. 
By Lemma~\ref{proyenveys} $\ga_{{\eta_j},\gamma}\ge -1$. 
Then 
$$\ga_{\gamma, \eta_j}=-\frac{d_{\eta_j}}{d_\gamma} \ga_{{\eta_j},\gamma}\le \frac{d_{\eta_j}}{d_{\gamma}}.$$
As $a_\gamma>0$, this implies that 
$$
\Delta_j=\frac{a_\gamma d_{\eta_j}}{d_{\gamma}}-\ga_{\gamma,\eta_j}\ge 0.$$
By Lemma~\ref{veyslemma} 
$$\left\lfloor\Delta_1\right\rfloor+ \sum_{j=2}^m(\left\lfloor\Delta_j\right\rfloor+1)
=w_\Gamma(\gamma)a_\gamma+1-\varphi
,$$
where 
$$\varphi=\sum_{j=1}^m\delta_j+\widehat d_\gamma\lambda(a_\gamma,\gamma)$$
is a nonnegative integer.

We can assume that for some $j$ there exists a nonnegative integer $a_{\eta_j}$ with $\lambda(a_{\eta_j},\eta_j) \le \lambda(a_\gamma,\gamma)$.
If this is not the case, then choose any nonnegative integers  $a_{\eta_j}$ satisfying
$$\sum_{j=1}^m a_{\eta_j}=w_\Gamma(\gamma)a_\gamma.$$
Then $\lambda(a_{\eta_j},\eta_j)>\lambda(a_\gamma,\gamma)$ for all $j=1,\ldots,m$ by the above assumption. But this means that
we have proven the claim.

Suppose thus that, for example, $\lambda(a_{\eta_1},\eta_1) \le \lambda(a_\gamma,\gamma)$. We will first consider the case 
$\lambda(a_{\eta_1},\eta_1) < \lambda(a_\gamma,\gamma)$. Because $\lambda(a_\gamma,\gamma)=\lambda(\Delta_{\eta_1},\eta_1)$,
this implies that $a_{\eta_1}<\Delta_1$. Then either $a_{\eta_1}\le \left\lfloor\Delta_1\right\rfloor-1$ 
or $a_{\eta_1}=\left\lfloor\Delta_1\right\rfloor$ and
$\delta_1>0$. In the first case we write
$$\left\lfloor\Delta_1\right\rfloor-1+\sum_{j=2}^m(\left\lfloor\Delta_j\right\rfloor +1)
=\sum_{j=1}^m\left\lfloor\Delta_j\right\rfloor+m-2 
\le w_\Gamma(\gamma)a_\gamma,$$
whereas in the latter case we have $\varphi\ge 1$ so that
$$\left\lfloor\Delta_1\right\rfloor+\sum_{j=2}^m(\left\lfloor\Delta_j\right\rfloor +1)
=\sum_{j=1}^m\left\lfloor\Delta_j\right\rfloor+m-1 
=w_\Gamma(\gamma)a_\gamma+1-\varphi\le w_\Gamma(\gamma)a_\gamma.
$$
It comes therefore out that it is possible to find numbers 
$a_{\eta_j} \ge \left\lfloor\Delta_j\right\rfloor+1>\Delta_j$ 
for $j=2,\ldots,m$ such that 
$$
\sum_{j=1}^m a_{\eta_j}=w_\Gamma(\gamma)a_\gamma.$$
Because $\lambda(a_\gamma,\gamma)=\lambda(\Delta_j,\eta_j)$, 
$a_{\eta_j}>\Delta_j$
implies 
$\lambda(a_{\eta_j},\eta_j)>\lambda(a_\gamma,\gamma)$. 

Consider then the case $\lambda(a_{\eta_1},\eta_1) = \lambda(a_\gamma,\gamma)$. Now $a_{\eta_1}=\Delta_1=\lfloor\Delta_1 \rfloor$.
We immediately observe that the above argument works if $\widehat d_1>0$. This is the case also if some $\delta_{j}
>0$. We can therefore assume that $\delta_2=0$. Then $\Delta_2=\lfloor \Delta_2 \rfloor$ is an integer.
Take $a_{\eta_2}=\Delta_2$. We can now write
$$
\left\lfloor\Delta_1\right\rfloor+\left\lfloor\Delta_2\right\rfloor+\sum_{j=3}^m(\left\lfloor\Delta_j\right\rfloor +1)
=\sum_{j=1}^m\left\lfloor\Delta_j\right\rfloor+m-2 
\le w_\Gamma(\gamma)a_\gamma-\varphi \le w_\Gamma(\gamma)a_\gamma.
$$
Finally note that $\lambda(a_{\eta_2},\eta_2) = \lambda(a_\gamma,\gamma)$ of course implies $a_{\eta_2}=\Delta_2$.
\end{proof}

\section{Main results}

We first want to give a criterium for a positive rational number to be a jumping number. We begin with
two lemmata. In the first one an antinef divisor is constructed for an ordered tree structure
on the dual graph:

\begin{lemma}\label{vieritys}
Let $S\subset |\Gamma|$ be a connected set of vertices. Suppose that there is a collection of 
nonnegative integers $\{a_\nu\in\bbN\mid d(\nu,S)\le1\}$ such that 
\begin{itemize}
\item[i)] $\lambda(a_\nu,\nu)>\xi=\lambda(a_\gamma,\gamma)$ for any $\gamma \in S$ and $\nu \sim S$; 
\item[ii)] $\displaystyle w_\Gamma(\gamma)a_{\gamma} \ge \sum_{\nu\sim\gamma}a_\nu$ for every $\gamma\in S$.
\end{itemize}
Then there exists an antinef divisor $F\in \Lambda$ such that 
\begin{itemize}
\item[1)] $f_{\nu}=a_{\nu}$ for all $\nu \in |\Gamma|$ with $d(\nu,S)\le1$;
\item[2)] $\widehat f_{\nu}=0$ for $\nu\not\in S$ unless $\nu$ is an end; 
\item[3)] For any $\nu\sim\mu$ such that $d(\nu,S)>d(\mu,S)$ we have 
$$\lambda(f_{\nu},\nu)>\lambda(f_{\mu},\mu).$$
\end{itemize}
\end{lemma}

\begin{proof} If $N=1$, then there is nothing to prove. Suppose thus that $N>1$. We will define nonnegative integers 
$f_\nu$ inductively on $d(\nu,S)$. First set $f_{\nu}=a_{\nu}$ for all $\nu\in |\Gamma|$ with $d(\nu,S)\le 1$. 
If $S=|\Gamma|$, we are done. Suppose 
then that $f_\nu$ has been defined for some $\nu\in |\Gamma|$ with $d(\nu,S)>0$. Because $S$ is connected and $\Gamma$ contains 
no loops, there is a unique $\mu' \in |\Gamma|$ such that $\mu' \sim \nu$ and 
$d(\mu',S)=d(\nu,S)-1$. By induction we know that $\lambda(f_\nu,\nu)>\lambda(f_{\mu'},{\mu'})$.
Therefore we can use Lemma~\ref{veys} to find for each $\mu'\not=\mu \sim \nu$ a nonnegative integer
$f_\mu$ such that $\lambda(f_\mu,\mu)>\lambda(f_\nu,\nu)$. In the case where $\nu$ is not an end we also get 
$$\widehat f_\nu=w_\Gamma(\nu)f_\nu-\sum_{\mu\sim\nu}f_\mu=0.$$
When all the numbers $f_\nu$ have so been defined, we can set $F=\sum_\nu f_\nu E_\nu.$ 
It remains to show that $F$ is antinef. This is equivalent to $\widehat f_\nu \ge 0$ for all $\nu$.
We have already seen that $\widehat f_\nu=0$ if $\nu\not\in S$ is not an end. Moreover, $\widehat f_\nu \ge 0$
when $\nu\in S$. In order to complete the proof we need to use the following Lemma~\ref{R}.
\end{proof}

\begin{lemma}\label{R}
Assume that $N>1$. Let $\tau$ be an end and $\mu$ the vertex adjacent to it. 
If $G\in\Lambda$ is such that $\lambda(g_\tau,\tau)>\lambda(g_\mu,\mu)$, then ${\widehat g}_\tau\ge0$.
\end{lemma}

\begin{proof}
Note first that ${\widehat g}_\tau=w_\Gamma(\tau)g_\tau-g_\mu$ by (\ref{P}).
The assumption now says that 
$$
\frac{g_\mu+k_\mu+1}{d_\mu}<
\frac{g_\tau+k_\tau+1}{d_\tau}.
$$ 
By the formulas $(\ref{P})$ and $(\ref{KP})$
$$
\frac{w_\Gamma(\tau)\left(g_\tau+k_\tau+1\right)}{w_\Gamma(\tau)d_\tau}=
\frac{w_\Gamma(\tau)g_\tau+k_\mu+2}{d_\mu+\widehat{d}_\tau}\le
\frac{(w_\Gamma(\tau)g_\tau+1)+k_\mu+1}{d_\mu}.
$$
Thereby 
$$
\frac{g_\mu+k_\mu+1}{d_\mu}<\frac{(w_\Gamma(\tau)g_\tau+1)+k_\mu+1}{d_\mu},
$$
so that
$g_\mu<w_\Gamma(\tau)g_\tau+1$, i.e., ${\widehat g}_\tau\ge0$.
\end{proof} 

We are now able to prove our first main result:

\begin{theorem}\label{hyppylukukriteeri}
Let $\fa$ be an ideal in a two-dimensional regular local ring $R$. A positive rational
number $\xi$ is a jumping number of $\fa$
if and only if there exists a connected set of vertices $S\subset|\Gamma|$ and a
collection of nonnegative integers $\{a_\eta\in\bbN\mid d(\eta,S)\le1\}$ 
satisfying the following conditions:
\begin{itemize}
\item[i)]   $\lambda(a_\eta,\eta)>\xi=\lambda(a_\gamma,\gamma)$ for every $\gamma\in S$
and every $\eta\sim S$;
\item[ii)]  $\displaystyle w_{\Gamma}(\gamma) a_\gamma=\sum_{\nu \sim\gamma}a_\nu$ for every vertex
$\gamma\in S$ (when $N>1$).
\end{itemize}
The condition ii) can be replaced by the condition
\begin{itemize}
\item[ii')] $\displaystyle w_{\Gamma}(\gamma) a_\gamma\ge \sum_{\nu\sim\gamma}a_\nu$ for every vertex
$\gamma\in S$.
\end{itemize}
When these conditions hold, there exists an antinef divisor $F$ with $\xi=\xi_F$ such that
$S=S_F$. Finally, consider the conditions 
\begin{itemize}
\item[1)] $S$ is a chain;
\item[2)] if $\gamma$ is not an end of $S$, then $\widehat d_\gamma=0$;
\item[3)] if $\gamma$ is an end of $S$, then $\gamma$ is a star or $\widehat d_\gamma>0$.
\end{itemize}
The set $S$ can then be chosen in such a way that conditions 1) and 2) hold while condition 3) is true for any $S$.
\end{theorem}

\begin{proof}

The case $N=1$ being trivial, we may assume that $N>1$.

Suppose first that there exists a connected set $S\subset|\Gamma|$ and a
collection of nonnegative integers $\{a_\nu\in\mathbb N | d(\nu, S)\le 1\}$
satisfying conditions i) and ii'). Let $F \in\Lambda$ be an antinef divisor
as in Lemma 4.1. For any chain $\gamma=\nu_0\sim\dots\sim\nu_r = \nu$ going away from $S$, where
$\gamma\in S$, we then have
$$
\lambda(f_\nu,\nu)>\dots>\lambda(f_\gamma,\gamma)=\lambda(a_\nu,\nu)=\xi
$$
so that $\xi=\min\{\lambda(f_\nu,\nu)|\nu\in|\Gamma|\}$. Then $\xi=\xi_F$ is
a jumping number by Proposition~\ref{parametrisointi}. Note that for an end vertex $\gamma$ of $S$, 
we have by Lemma~\ref{ketju} either $\widehat d_\gamma>0$ or $v_\Gamma(\gamma)\ge3$, which shows that 
condition 3) automatically holds for $S$.

Conversely, if $\xi$ is a jumping number, then by Proposition 2.1 there exists
an antinef divisor $F$ such that $\xi=\xi_F$. We may now choose a connected
component $S$ of $S_F$ and set $a_\nu:=f_\nu$ for every $\nu$ with
$d(\nu,S)\le1$. Clearly, the conditions i) and ii') are satisfied. 

It remains to show that if there is a connected set $S'\subset|\Gamma|$ with a
collection $\{a'_\nu\in\mathbb N|d(\nu,S')\le1\}$ that satisfies conditions
i) and ii'), then we can find a new set $S\subset S'$ together with a
collection $\{a_\nu\in\mathbb N|d(\nu,S)\le1\}$ satisfying conditions i) and
ii). Moreover, also conditions 1) and 2) should hold for the set $S$.  

If there is a vertex $\gamma\in S'$ which has only one adjacent vertex, say $\eta$, and 
$w_\Gamma(\gamma)a'_\gamma>a'_\eta$, then we choose 
$S:=\{\gamma\}$ and take $a_\gamma:=a'_\gamma$, $a_\eta:=w_\Gamma(\gamma)a'_\gamma$. 

When this is not the case, we look for chains consisting of vertices $\gamma\in S'$ such that
$\widehat d_\gamma=0$ if $\gamma$ is not an end of the chain. We now take our set $S$ to be any 
maximal chain of this type. Obviously conditions 1) and 2) then hold. 

We will next define the integers $a_\nu$ for all 
$\nu\in|\Gamma|$ with $d(\nu,S)\le1$. To begin with, set $a_\nu:=a'_\nu$ if $\nu\in S$. Then take 
any $\gamma\in S$ and look at the vertices $\eta\sim\gamma$. Note that if there is only one
vertex $\eta$ adjacent to $\gamma$, then we now necessarily have $w_\Gamma(\gamma)a'_\gamma=a'_\eta$.
By the maximality of $S$, $\eta\in S'$ implies $\eta\in S$ so that $a_\eta:=a'_\eta$ will do. 
Thus we may assume $v_\Gamma(\gamma)\ge2$.

Suppose first that $\gamma$ is an end of $S$ such that $\lambda(a'_\eta,\eta)>\xi$ for all $\gamma\sim\eta\notin S$. By assumption
ii') we have
$$
w_\Gamma(\gamma) a_\gamma\ge\sum_{\gamma\sim\eta\in
S}a_\eta+\sum_{\gamma\sim\eta\notin S}a'_\eta.
$$
If the equality holds, we take $a_\eta:=a'_\eta$ for every $\gamma\sim\eta\notin S$.
When this is not the case, we may increase the $a'_\eta$:s, $\eta \notin S$, to find numbers 
$a_\eta$ such that
$$
w_\Gamma(\gamma) a_\gamma=\sum_{\eta\sim\gamma}a_\eta.
$$
For any $\eta\sim\gamma$, $\eta\notin S$, we then have
$\lambda(a_\eta,\eta)\ge\lambda(a'_\eta,\eta)>\xi$.

Otherwise, we can utilize Lemma~\ref{veys} to define the numbers $a_\eta$.
For this we note that if $\gamma$ now is an end of $S$, then necessarily, by
the maximality of $S$, we have $\widehat d_\gamma>0$ and there must be a vertex
$\eta_1\in S$ adjacent to $\gamma$. Furthermore, when $\gamma$ is not an end of 
$S$, there are two vertices $\eta_1, \eta_2 \in S$ adjacent to $\gamma$. 
\end{proof}

\begin{remark}\label{hyppylukukriteerihuomautus}
If $\xi=\xi_F$ where $F$ is any antinef divisor, then it comes out from the proof of Theorem~\ref{hyppylukukriteeri} that we can construct the set 
$S$ in such a way that $S\subset S_F$ and that $a_{\gamma}=f_{\gamma}$ for all $\gamma\in S$. 

\end{remark}

Recall that the first jumping number of $\fa$ is the log canonical threshold 
$$\lct(\fa)=\min_\nu \frac{k_\nu+1}{d_\nu}.$$

\begin{corollary}\label{logcanonical}
Consider the set of those vertices $\gamma\in |\Gamma|$ for which 
$$\lct(\fa)=\frac{k_\gamma+1}{d_\gamma}.$$ This set then satisfies the 
conditions 1) -- 3) of Theorem~\ref{hyppylukukriteeri}.
\end{corollary}

\begin{proof}Recall first that $f_\gamma>0$ for every $\gamma\in |\Gamma|$ if $0\not=F\in \Lambda$ is an antinef
divisor. The zero divisor $0$ is then the unique antinef divisor $F\in \Lambda$ such that $\lct(\fa)=\xi_F$. 
Therefore the claim follows from Theorem~\ref{hyppylukukriteeri}.
\end{proof}

\begin{remark}This
was observed by Veys in~\cite[Theorem 3.3]{V} in the context of an embedded resolution of a curve. 
\end{remark}

\begin{example}\label{esimerkki1} Conditions 1) and 2) of Theorem~\ref{hyppylukukriteeri} do not hold for $S_F$ for 
an arbitrary antinef divisor $F\in \Lambda$. To see this, consider the configuration of exceptional divisors described
by the dual graph 
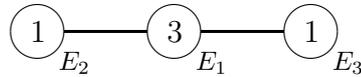
\newline
\unitlength=1,2mm
\begin{picture}(70,22)(-19,-11)
\put(20,0){\circle{6}}
\put(35,0){\circle{6}}
\put(50,0){\circle{6}}
\put(23,0){\line(1,0){9}}
\put(38,0){\line(1,0){9}}
\put(20,0){\makebox(0.2,0){$1$}}
\put(35,0){\makebox(0.2,0){$3$}}
\put(50,0){\makebox(0.2,0){$1$}}
\begin{footnotesize}
\put(39,-3.5){\makebox(0.2,0){$E_1$}}
\put(54,-3.5){\makebox(0.2,0){$E_3$}}
\put(24,-3.5){\makebox(0.2,0){$E_2$}}
\end{footnotesize}
\end{picture}
\newline
and the proximity matrix
$$P=\hspace{-3pt}\left[\begin{array}{rrr}1&0&0\\^-1&1&0\\^-1&0&1\end{array}\right].$$
The dual graph is then associated to a minimal principalization of the ideal $\fa=\fp_2\fp_3$.
Now 
$$D = 2E_1 + 3E_2 + 3E_3=\widehat E_2+\widehat E_3\quad\hbox{and}\quad K= E_1 + 2E_2 + 2E_3.$$
Let us take
$$F:= E_1 + E_2 + E_3= \widehat E_1.$$
One now calculates that
$$\lambda(f_1,1)=\frac{3}{2},\lambda(f_2,2)=\lambda(f_3,3)=\frac{4}{3}.$$
Then $\xi_F=\frac{4}{3}$ with $S_F=\{2, 3\}$. In particular, we see that $S_F$ is disconnected. However, if
$$F':= E_1 + E_2 + 2E_3=\widehat E_3,$$
then also $\xi_{F'}=\frac{4}{3}$, but $S_{F'}=\{2\}$. 
We now observe that $S_{F'}$ satisfies the conditions 1) -- 3) of Theorem~\ref{hyppylukukriteeri}.
\end{example}

Theorem~\ref{hyppylukukriteeri} implies that a support of a jumping number always contains star vertices or
vertices corresponding to Rees valuations. In Corollary~\ref{hyppyluvut} we are now going to show the converse: this kind of
vertices always support some jumping number. First we need two lemmata: 

\begin{lemma}\label{läheisyyslemma}Let $\gamma \in |\Gamma|$. If $\nu_1 \prec \cdots \prec\nu_r=\eta$ are points proximate to $\gamma$, 
then $d_\eta>rd_\gamma$.

\end{lemma}

\begin{proof}For every $i=1,\ldots,r$, we have $\nu_i \succ \gamma$ and $\nu_i\succ \nu_{i-1}$, where $\nu_0=\gamma$. 
Using the base change formula (\ref{BC}), we get
$$d_{\nu_i}=d_{\nu_{i-1}}+d_\gamma + d^*_{\nu_i}>d_{\nu_{i-1}}+d_\gamma$$ so that
$$
d_\eta>d_{\nu_{r-1}}+d_\gamma>d_{\nu_{r-2}}+2d_\gamma >\cdots>
d_{\nu_1}+(r-1)d_\gamma>rd_\gamma.$$
\end{proof}

\begin{lemma}\label{epäyhtälölemma}Let $\gamma\in |\Gamma|$.
\begin{itemize}
\item[a)] If $\widehat d_\gamma>0$, then $d_\gamma\ge k_\gamma\widehat d_\gamma$.
\item[b)] If $v_{\Gamma}(\gamma)+\widehat d_\gamma\ge 3$, then $d_\gamma-k_\gamma\ge 2$. 
\end{itemize}
\end{lemma}

\begin{proof}By (\ref{BC}) we have $d^*_\nu=\sum_\mu \widehat d_\mu q_{\mu,\nu}\ge \widehat d_\gamma q_{\gamma,\nu}$ for every $\nu\in |\Gamma|$.
Again by (\ref{BC}), and (\ref{K}), we then obtain
$$
d_\gamma-k_\gamma\widehat d_\gamma=\sum_\nu(d^*_\nu-\widehat d_\gamma)q_{\gamma,\nu}
\ge
\widehat d_\gamma\sum_\nu(q_{\gamma,\nu}-1)q_{\gamma,\nu}\ge0
$$
proving a).

In order to prove b), we first observe that
$d^*_1=d_1=v_1(\fa)\ge v_1(\fp_\gamma)=q_{\gamma,1}\ge1$ and $d^*_1\ge
d^*_\gamma\ge\widehat d_\gamma$. If now $q_{\gamma,1}\ge2$ or if $\widehat
d_\gamma\ge3$, we get \begin{equation*}d_\gamma-k_\gamma =\sum_\nu(d^*_\nu-1)q_{\gamma,\nu}\ge (d^*_1-1)q_{\gamma,1}\ge 2,\end{equation*}
as wanted. 

Consider then the case $q_{\gamma,1}=1$ and $\widehat d_\gamma\le2$. The
formula (\ref{Q}) implies that in this case $\gamma$ is proximate to at 
most one adjacent vertex. The number of adjacent vertices proximate to $\gamma$ 
must then be $v_\Gamma(\gamma)-t$, where $t\in \{0,1\}$. 
By using (3), we get 
$$ 
d^*_\gamma
= \widehat d_\gamma + \sum_{\nu\succ\gamma}d^*_\nu 
\ge \widehat d_\gamma + v_\Gamma(\gamma)-t
\ge 3-t.
$$ 
Assuming $\gamma\neq 1$, i.e., $t=1$, we now obtain
$$d_\gamma-k_\gamma\ge (d_1^* - 1)+ (d_\gamma^*-1) \ge 2.$$
In the case $\gamma=1$ 
we have $t=0$, and thus $d^*_\gamma\ge3$.
So $d_\gamma-k_\gamma\ge d^*_\gamma-1\ge2$.

\end{proof}

We are now ready to prove the promised result:

\begin{corollary}\label{hyppyluvut}Let $\gamma \in |\Gamma|$.
\begin{itemize}
\item[a)]
If $\widehat d_\gamma$ is positive, then 
$$
\frac{n}{d_\gamma}
$$
is a jumping number of $\fa$ with a support $\{\gamma\}$ for any integer $n$
such that $n\widehat d_\gamma>d_\gamma$. In particular, this implies that
$$
1+\frac{n}{d_\gamma}
$$ is a jumping number of $\fa$ with a support $\{\gamma\}$ for every positive integer $n$.

\item[b)] If $v_\Gamma(\gamma)+\widehat d_\gamma\ge 3$, then  
$$
1-\frac{1}{d_\gamma}
$$
is a jumping number of $\fa$ with a support $\{\gamma\}$. This is the case
especially when $\gamma$ is a star.

\end{itemize}
\end{corollary}

\begin{proof} 
Suppose that $\{\eta\mid\eta\sim\gamma\}=\{\eta_1,\dots,\eta_m\}$. Let
$a_{\gamma}$ be a nonnegative integer. By Lemma~\ref{veyslemma} we now have
$$
\sum_{j=1}^m(\lfloor\Delta_j \rfloor+1)=w_\Gamma(\gamma)a_\gamma+2-\varphi.
$$
where 
$$\varphi=\sum_{j=1}^m\delta_j+\widehat d_\gamma\lambda(a_\gamma,\gamma)$$
is a nonnegative integer.  
It follows that if $\varphi\ge2$, then we can find nonnegative integers $a_{\eta_j}
\ge \lfloor\Delta_j \rfloor+1
$ for
$j=1,\ldots,m$ such that
$$
\sum_{j=1}^ma_{\eta_j} =w_\Gamma(\gamma)a_\gamma$$
and that 
$$\lambda(a_{\eta_j},\eta_j)\ge \lambda(\lfloor\Delta_j \rfloor+1, \eta_j) >\lambda(\Delta_j, \eta_j)=\lambda(a_\gamma,\gamma).$$
This means that we can use Theorem~\ref{hyppylukukriteeri} to conclude that $\xi=\lambda(a_\gamma,\gamma)$ is a jumping number 
with a support $\{\gamma\}$. Let us now show that in both cases a) and b) we can choose $a_\gamma$ in such a way that
$\varphi\ge 2$.

Consider first the case a). For any integer $n$ such that $n\widehat d_\gamma>d_\gamma$, set
$$a_\gamma=n-k_\gamma-1.$$
By Lemma~\ref{epäyhtälölemma} a) $a_\gamma$ is positive. Moreover
$$
\lambda(a_\gamma,\gamma)=\frac{n}{d_\gamma}.$$
Then $\varphi \ge\left\lceil\widehat
d_\gamma\lambda(a_\gamma,\gamma)\right\rceil\ge2$ as wanted.

Consider then the case b). Set $a_\gamma=d_\gamma-k_\gamma-2$. 
Now Lemma~\ref{epäyhtälölemma} b) implies that $a_\gamma$ is nonnegative. 
Furthermore,
$$
\lambda(a_\gamma,\gamma)=1-\frac{1}{d_\gamma}.
$$
Then 
$$
\Delta_j=d_{\eta_j}\lambda(a_\gamma,\gamma)-
k_{\eta_j}-1=-\frac{d_{\eta_j}}{d_\gamma}+d_{\eta_j}-k_{\eta_j}-1.
$$
for all $j=1,\ldots,m$.
Thus
$$
\delta_j=\Delta_j - \lfloor \Delta_j \rfloor 
=\left\lceil\frac{d_{\eta_j}}{d_\gamma}\right\rceil-\frac{d_{\eta_j}}{d_\gamma}.
$$
By the formula $(\ref{P})$ 
$$w_\Gamma(\gamma)d_\gamma=\sum_{j=1}^md_{\eta_j}+{\widehat d}_\gamma.$$
Hence
$$
\sum_{j=1}^m\delta_j
=\sum_{j=1}^m\left\lceil\frac{d_{\eta_j}}{d_\gamma}\right\rceil-\sum_{j=1}^m\frac{d_{\eta_j}}{d_\gamma}
=\sum_{j=1}^m\left\lceil\frac{d_{\eta_j}}{d_\gamma}\right\rceil-w_\Gamma(\gamma)+\frac{\widehat d_\gamma}{d_\gamma}
$$
so that 
$$
\varphi=\sum_{j=1}^m\delta_j+\widehat d_\gamma\lambda(a_\gamma,\gamma)=
\sum_{j=1}^m\left\lceil\frac{d_{\eta_j}}{d_\gamma}\right\rceil-w_\Gamma(\gamma)+\widehat d_\gamma.
$$
Since $\gamma \sim \eta_j$, we have either $\gamma\prec \eta_j$ or $\eta_j \prec \gamma$. In the
first case we have a maximal chain $\nu_1\prec \cdots\prec \nu_{r_j}=\eta_j$ of points
proximate to $\gamma$. By Lemma~\ref{läheisyyslemma} this 
means that $d_{\eta_j}\ge r_j d_\gamma+1$. This certainly holds also when $\eta_j \prec \gamma$ if we set 
$r_j=0$ in this case.
Therefore we have, for every
$j=1,\dots,m$, 
$$
\left\lceil\frac{d_{\eta_j}}{d_\gamma}\right\rceil\ge1+r_j
\hspace{4pt}\text{ and }\hspace{4pt}
\sum_{j=1}^m r_j
=\#\{\nu\in |\Gamma|\mid \nu \succ \gamma\}=w_\Gamma(\gamma)-1.
$$
It follows that
$$
\sum_{j=1}^m\left\lceil\frac{d_{\eta_j}}{d_\gamma}\right\rceil\ge
m-1+w_\Gamma(\gamma),
$$
and further,
$
\varphi\ge m-1+\widehat d_\gamma\ge2,
$
which proves the claim.
\end{proof}

\begin{remark}In the case of a curve on a smooth surface, it was observed by Smith and Thompson in~\cite[Theorem 3.1]{ST}
that b) of Corollary~\ref{hyppyluvut} holds for any star vertex.
\end{remark}

We will next show how new jumping nunbers can be obtained from a given one. 

\begin{proposition}\label{uudet}
Suppose $\xi$ is a jumping number of $\fa$ with a 
support $S \subset|\Gamma|$. Write $$d=\gcd\{d_\eta|d(\eta,S)\le1\}.$$ 
Then, for any $n\in\mathbb N$,
$$\xi+\frac{n}{d}$$ is also a jumping number of $\fa$ with a support $S$.
\end{proposition}

\begin{proof}

Suppose first that $S$ is connected.
Let $\{a_\eta|d(\eta,S)\le1\}$ be a collection of nonnegative integers
satisfying the conditions of Theorem~\ref{hyppylukukriteeri}. For every 
vertex $\eta\in |\Gamma|$ with $d(\eta,S)\le1$,
write 
$$
b_\eta:=\frac{d_\eta}{d}.
$$
Clearly, 
$$
\lambda(a_\eta+nb_\eta,\eta)=\lambda(a_\eta,\eta)+\frac{n}{d}
$$
for any $\eta\in |\Gamma|$ and $n\in\bbN$. Hence
$$
\lambda(a_\eta+nb_\eta,\eta)>\xi+\frac{n}{d}=\lambda(a_\gamma+nb_\gamma,\gamma)
$$
for every $\gamma\in S$ and $\eta\sim S$. By the proximity inequality for $\fa$ 
$$
w_\gamma b_\gamma d=
w_\gamma d_\gamma\ge\sum_{\nu\sim\gamma}d_\nu
=d\sum_{\nu\sim\gamma}b_\nu.
$$
We thus see that 
$$
w_\gamma (a_\gamma+nb_\gamma)\ge\sum_{\nu\sim\gamma}(a_\nu+nb_\nu).
$$
for every $\gamma\in S$. The claim then results from Theorem~\ref{hyppylukukriteeri}. 

Consider then the general case. Let $S_1,\dots,S_r$ be the connected components of $S$. Write
$$d_i:=\gcd\{d_\eta|d(\eta,S_i)\le1\}\quad (i=1,\ldots,r).$$ Then clearly $d=\gcd\{d_1,\dots,d_r\}$. 
By the preceeding case we now know that $$\xi+\frac{n}{d}$$ is a jumping number with a support $S_i$ 
for every $i=1,\dots,r$. Then the claim follows from the following Lemma~\ref{yhdiste}:

\end{proof}

\begin{lemma}\label{yhdiste}If $S_1,S_2\subset |\Gamma|$ are supports of a jumping number $\xi$, then so is $S_1\cup S_2$. 
\end{lemma}

\begin{proof} Let $F_1, F_2\in \Lambda$ be antinef divisors such that
$S_1=S_{F_1}$ and $S_2=S_{F_2}$. Let us define a divisor $F$ by setting
$$f_\nu:=\min\{f_{1,\nu},f_{2,\nu}\}.$$
Since $F_1$ and $F_2$ are both antinef, we see that
$$\sum_{\eta\sim\gamma}f_\eta\le
\sum_{\eta\sim\gamma}f_{i,\eta}\le
w_\Gamma(\gamma)f_{i,\gamma}
$$ for any $\gamma\in|\Gamma|$ and $i=1,2$, and further,
$$
\sum_{\eta\sim\gamma}f_\eta\le w_\Gamma(\gamma)f_\gamma.
$$
Hence $F$ is antinef, too. Clearly, $\lambda(f_\nu,\nu)\ge\xi$, where the
equality holds if and only if $\nu\in S_1\cup S_2$. Therefore $S_1\cup S_2=S_F$, which proves
the claim.
\end{proof}

Tucker introduced in \cite{T} the notions of contibution and critical contribution of a jumping number by a divisor. The contribution of
a prime divisor had earlier been defined by Smith and Thompson in \cite{ST}. Consider a reduced subdivisor $$G=E_{\gamma_1}+\ldots +E_{\gamma_r}\le D.$$ 
One first says that a positive rational number $\xi$ is a \textit{candidate jumping number} for $G$ if $\xi d_\gamma$ is an integer for all $\gamma \in \{\gamma_1,\ldots,\gamma_r\}$. The divisor $G$ is then said to \textit{contribute} the jumping number $\xi$ if $\xi$ is a candidate jumping number for $G$ and $$\cJ(\fa^\xi) \subsetneq \Gamma(X,\cO_X(G+K-\lfloor \xi D\rfloor)).$$ Finally, a contribution is said to be \textit{critical} if, in addition, no proper subdivisor of G contributes $\xi$, i.~e., $$\cJ(\fa^\xi) = \Gamma(X,\cO_X(G'+K-\lfloor\xi D\rfloor))$$
for all divisors $0\le G'< G$.

\begin{proposition}\label{kriittinenlemma}
If a divisor $G$ critically contributes $\xi$, then $\xi=\xi_F$ for an antinef divisor $F$ with $S_F=|G|$. 
\end{proposition}

\begin{proof}
Set $F:=(\left\lfloor\xi D \rfloor-K-G\right\rfloor)^\sim$. 
Because $\xi$ is a candidate jumping number for $G$, we have 
$$G+K-\lfloor \xi D\rfloor \le K-\lfloor(\xi-\epsilon)D\rfloor$$
for small enough $\epsilon>0$. Let $\fb=\Gamma(X,\cO(-F))$ be the complete ideal associated to $F$. By
the above we get $\fb \subset \cJ(\fa^{\xi-\epsilon})$. Clearly $\cJ(\fa^\xi) 
\subset \fb$, but as $G$ is contributing, we must have $\cJ(\fa^\xi)\not=\fb$. 
It then follows from Proposition~\ref{parametrisointi} that $\xi=\xi_F$. Let us now
show that $\gamma\in |G|$ is equivalent to $f_\gamma=\xi d_\gamma-k_\gamma-1$. 
Since $\xi=\xi_F$, we have $$f_\gamma\ge  \xi d_\gamma-k_\gamma-1$$ 
for all $\gamma\in |\Gamma|$. Now $\lfloor\xi D \rfloor-K-G \le F$. For $\gamma\not\in |G|$, this means that 
$$f_\gamma \ge \lfloor \xi d_\gamma \rfloor -k_\gamma> \xi d_\gamma -k_\gamma-1.$$
Suppose that we would have $f_\gamma\ge \xi d_\gamma-k_\gamma$ for some $\gamma\in |G|$. Write $G':=G-E_\gamma$. 
Then $G'$ is a proper subdivisor of $G$ and 
$$
\lfloor \xi D \rfloor -K-G\le
\lfloor \xi D \rfloor-K-G'\le F
=(\lfloor \xi D \rfloor-K-G)^\sim
$$
implying that $(\lfloor \xi D \rfloor -K-G')^\sim=F$. But then $$\cJ(\fa^\xi) \subsetneq \Gamma(X,\cO(K+G'-\lfloor \xi D \rfloor ))$$
which is impossible, because $G$ critically contributes $\xi$.
\end{proof}

\begin{lemma}\label{kontribuutiolemma}Suppose that $F$ is an antinef divisor such that $\xi=\xi_F$. Then
$$G:=\sum_{\gamma \in S_F}E_{\gamma}$$ is a contributing divisor for $\xi$.
\end{lemma}

\begin{proof}Because $\lambda(f_{\gamma},\gamma)=\xi$ for all $\gamma\in S_F$, $\xi$ is a candidate jumping number for $G$.  
Now 
$$
f_\nu \ge \xi d_\nu-k_\nu-1
$$
where the equality holds if and only if $\nu\in S_F$. In other words,
$$
f_\nu \ge \lfloor \xi d_\nu \rfloor -k_\nu
$$
if $\nu\not\in S_F$ whereas
$$
f_\nu = \lfloor \xi d_\nu \rfloor-k_\nu-1
$$
if $\nu\in S_F$.
But this says that 
$$-F \le G+K-\lfloor\xi D\rfloor$$
implying that
$$
\Gamma(X,\cO(-F)) \subset\Gamma(X,\cO_X(G+K-\lfloor\xi D\rfloor )).$$
Because $\xi=\xi_F$, it then follows from Proposition~\ref{parametrisointi} that we cannot have
$$\Gamma(X,\cO_X(G+K-\lfloor\xi D\rfloor ))=\cJ(\fa^\xi).$$
\end{proof}

The following theorem now connects our approach to that of Smith, Thompson and Tucker:

\begin{theorem}\label{kriittinenlause}
Let $\fa$ be an ideal in a two-dimensional regular local ring $R$.
Let $G$ be a divisor critically contributing a jumping number $\xi$ of $\fa$.
Then there exists an antinef divisor $F$ such that $\widehat f_\gamma=0$ unless $\gamma$ is an end not in $|G|$, 
$\lambda(a_{\gamma},\gamma)$ reaches its minimum $\xi$ exactly when $\gamma\in |G|$ and grows strictly on 
every path away from $|G|$. Moreover, 
\begin{itemize}
\item[1)] $|G|$ is a chain;
\item[2)] if $\gamma\in |G|$ has more than one adjacent vertex in $|G|$ then $\widehat d_\gamma=0$;
\item[3)] if $\gamma$ is an end of $|G|$ then $\gamma$ is a star or $\widehat d_\gamma>0$.
\end{itemize}

\end{theorem}

\begin{proof}
According to Proposition~\ref{kriittinenlemma} there is an antinef divisor $F'$ such that $\lambda(f'_{\gamma},\gamma)=\xi$ exactly 
when $\gamma\in |G|$. Let $F$ be an antinef divisor as in Theorem~\ref{hyppylukukriteeri} such that  
$$S_F:=\{\gamma\in\Gamma\mid \lambda(f_{\gamma},\gamma)=\xi\}\subset \{\gamma\in\Gamma\mid \lambda(f'_{\gamma},\gamma)=\xi\}=|G|$$
(see Remark~\ref{hyppylukukriteerihuomautus}).
By Lemma~\ref{kontribuutiolemma} the divisor $\sum_{\nu\in S_F}E_\nu$ is contributing. As $G$ is critically
contributing, this implies that $S_F=|G|$, and we are done.
\end{proof}

\begin{remark}It was discovered by Tucker in~\cite[Theorem 5.1]{T} that if $G$ is a divisor critically contributing a jumping number, then 
$|G|$ satisfies conditions 1) -- 3). 

\end{remark}

Unfortunately, the existence of an antinef divisor $F$ as in Theorem~\ref{kriittinenlause} does not guarantee that the contribution is critical. 
This will come out from the following example:

\begin{example}\label{esimerkki2}Consider the configuration of exceptional divisors described by the dual graph
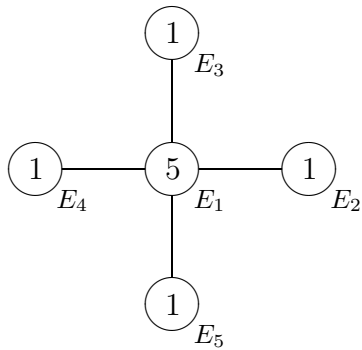
\newline
\unitlength=1,2mm
\begin{picture}(70,42)(-19,-21)
\put(20,0){\circle{6}}
\put(35,0){\circle{6}}
\put(35,-15){\circle{6}}
\put(35,15){\circle{6}}
\put(50,0){\circle{6}}
\put(23,0){\line(1,0){9}}
\put(35,-3){\line(0,-1){9}}
\put(35,3){\line(0,1){9}}
\put(38,0){\line(1,0){9}}
\put(20,0){\makebox(0.2,0){$1$}}
\put(35,0){\makebox(0.2,0){$5$}}
\put(35,-15){\makebox(0.2,0){$1$}}
\put(35,15){\makebox(0.2,0){$1$}}
\put(50,0){\makebox(0.2,0){$1$}}
\begin{footnotesize}
\put(39,-3.5){\makebox(0.2,0){$E_1$}}
\put(54,-3.5){\makebox(0.2,0){$E_2$}}
\put(39,11.5){\makebox(0.2,0){$E_3$}}
\put(24,-3.5){\makebox(0.2,0){$E_4$}}
\put(39,-18.5){\makebox(0.2,0){$E_5$}}
\end{footnotesize}
\end{picture}
\newline
and the proximity matrix
$$P=\hspace{-3pt}\left[\begin{array}{rrrrr}1&0&0&0&0\\^-1&1&0&0&0\\^-1&0&1&0&0\\^-1&0&0&1&0\\^-1&0&0&0&1\end{array}\right].$$
They are associated to the minimal principalization of the ideal $\fa=\fp_2^4\fp_3^4\fp_4^5\fp_5^7$. 
Now 
$$D= 20E_1 + 24E_2 + 24E_3 + 25E_4+  27E_5$$
and
$$K= E_1 + 2E_2 + 2E_3 + 2E_4+  2E_5.$$
Let us take
$$F:= 3E_1 + 3E_2 + 3E_3 + 4E_4+  5E_5= \widehat E_4+2\widehat E_5.$$
One then calculates that
$$\lambda(f_1,1)=\lambda(f_2,2)=\lambda(f_3,3)=\frac{1}{4},\lambda(f_4,4)=\frac{7}{25},\lambda(f_5,5)=\frac{8}{27}.$$
Thus $\xi_F=\frac{1}{4}$ and $S_F=\{1, 2, 3\}$. Set $G=E_1+E_2+E_3$. We now observe that the divisor $F$ satisfies 
the conditions of Theorem~\ref{kriittinenlause}. However, $G$ is not critically contributing. Indeed, set 
$$
F':= 3E_1 + 4E_2 + 3E_3 + 4E_4+  4E_5=\widehat E_2+\widehat E_4+\widehat E_5$$
and 
$$F'':= 3E_1 + 3E_2 + 4E_3 + 4E_4+  4E_5=\widehat E_3+\widehat E_4+\widehat E_5.$$
Then
$$\lambda(f'_1,1)=\frac{5}{20},\lambda(f'_2,2)=\frac{7}{24},\lambda(f'_3,3)=\frac{6}{24},\lambda(f'_4,4)=\frac{7}{25},\lambda(f'_5,5)=\frac{7}{27}$$
and
$$\lambda(f''_1,1)=\frac{5}{20},\lambda(f''_2,2)=\frac{6}{24},\lambda(f''_3,3)=\frac{7}{24},\lambda(f''_4,4)=\frac{7}{25},\lambda(f''_5,5)=\frac{7}{27}.$$
Therefore $\xi_{F'}=\xi=\xi_{F''}$ but 
$$
S_{F'}=\{1, 3\}\subsetneq S_F \supsetneq \{1,2\}=S_{F''}.
$$
By Lemma~\ref{kontribuutiolemma} both $S_{F'}$ and $S_{F''}$ contain subsets corresponding to divisors contributing $\xi$, but then the divisor $G$ cannot be  
critically contributing. Moreover, it is easy to see that both 
$E_{1}+E_{2}$ and $E_{1}+E_{3}$ critically contribute $\xi$. Note that their intersection 
$E_{1}$ does not contribute $\xi$. Observe also that in this example the set $S_F$ includes a  star vertex which is a non-end of $S_F$.

\end{example}

We next give a combinatorical criterium for the critical contribution:

\begin{theorem}\label{kriittinenkriteeri}Let $\fa$ be an ideal in a two-dimensional regular local ring $R$.
Let $G$ be a reduced subdivisor of $D$ such that $|G|$ contains at least two vertices. Then $G$ critically contributes a jumping number 
$\xi$ if and only if $|G|$ is connected and there exists a collection of nonnegative integers $\{a_\nu\in\bbN\mid d(\nu,|G|)\le1\}$ 
such that 
\begin{itemize}
\item[i)]$\lambda(a_\nu,\nu)>\xi=\lambda(a_\gamma,\gamma)$ for any $\gamma \in |G|$ and $\nu \sim |G|$; 

\item[ii)]$\displaystyle w_\Gamma(\gamma)a_{\gamma}=\sum_{\nu\sim\gamma}a_\nu$ for every $\gamma\in |G|$;

\item[iii)]$\lambda(a_\nu-1,\nu)\le\xi$ for any vertex $\nu\sim |G|$.
\end{itemize}
Moreover, it is enough that iii) holds for vertices $\nu\sim \gamma$, where $\gamma\in |G|$ is a star or $\widehat d_{\gamma}>0$.

Finally, if $G$ critically contributes $\xi$ and $F$ is any antinef divisor such that $\xi$ has support $|G|$ with respect to $F$, 
then we may take $a_\nu=f_\nu$ for all $\nu$ with $d(\nu,|G|)\le1$.

\end{theorem}

\begin{proof} Suppose first that $|G|$ is connected and that there is a collection of nonnegative integers $\{a_\nu\in\bbN\mid d(\nu,|G|)\le1\}$
satisfying conditions i) -- iii). Let $F\in \Lambda$ be an antinef divisor as in Lemma~\ref{vieritys}. This means that $$S_F=\{\gamma\in |\Gamma|\mid \lambda
(f_\gamma,\gamma)=\xi\}=|G|.$$ It then follows from Lemma~\ref{kontribuutiolemma} that the divisor $G$ is contributing. For some subset $S\subset |G|$ the 
divisor $\sum_{\nu\in S}E_\nu$ is necessarily critically contributing. If $S\not=|G|$, then there must be a vertex $\eta\in |G|\smallsetminus S$ adjacent to a vertex $\gamma\in S$. By Theorem~\ref{kriittinenlause} we can find an antinef divisor $F'$ such that $\xi=\xi_{F'}$ and $S=S_{F'}$. Then $\lambda(f'_{\eta}, \eta)>\xi=\lambda(a_\eta,\eta)$ so that $f'_\eta>a_\eta$. 
Because $\widehat{f_\gamma'}=0$, we have  
$$\sum_{\nu\sim\gamma}f'_\nu=w_\Gamma(\gamma)f'_{\gamma}=w_\Gamma(\gamma)a_{\gamma}=\sum_{\nu\sim\gamma}a_\nu.$$
Thus there must be at least one vertex $\mu\sim\gamma$ for which $f'_\mu<a_\mu$. So $\lambda(f'_\mu,\mu)<\lambda(a_\mu,\mu)$.
This implies $\mu\not\in |G|$ as otherwise $\lambda(f'_\mu,\mu)<\xi$ which
is impossible. If $\gamma$ is adjacent to some vertex in $S$, then $\gamma$ must be star. If $\gamma$ is not a star, then by 
Theorem~\ref{kriittinenlause} $S=\{\gamma\}$ and $\widehat d_{\gamma}>0$. Condition iii) now implies that 
$$\lambda(f'_\mu,\mu)\le\lambda(a_\mu-1,\mu)\le \xi.$$ However, this is possible only if 
$\lambda(f'_\mu,\mu)=\xi$. But this means that $\mu\in S$, which is a contradiction. Therefore we must have 
$S=|G|$ as wanted.

Suppose then that $G$ critically contributes $\xi$. Let $\gamma\in |G|$. By Theorem~\ref{kriittinenlemma} there then exists 
an antinef divisor $F$ such that $\lambda(f_\nu,\nu)>\xi=\lambda(f_\gamma,\gamma)$ for any $\nu\notin |G|$. 
Set $a_\nu=f_\nu$ for all $\nu \in |\Gamma|$ with $d(\nu,|G|)\le1 $.  
Suppose that there would exist a vertex $\eta\sim\gamma$ such that 
$\lambda(a_\eta-1,\eta)>\lambda(a_\gamma,\gamma)$. 
We may choose a vertex $\gamma'\in |G|$ such that $\gamma'\sim \gamma$. Let $S$ be a connected component of $|G|\setminus \{\gamma'\}$ containing $\gamma$. Choose 
a collection of nonnegative integers $\{a'_\nu\in\bbN\mid d(\nu,S)\le1\}$ so that $a'_\eta=a_\eta-1$ and 
$a'_{\gamma'}=a_{\gamma'}+1$ while $a'_\nu=a_\nu$ otherwise. Clearly $\lambda(a'_\nu,\nu)>\xi$ for any $\nu\sim S$. Also 
$$w_\Gamma(\gamma)a'_{\mu}=\sum_{\nu\sim\mu}a'_\nu$$ for every $\mu\in S$. Therefore we can use Lemma~\ref{vieritys} to 
find an antinef divisor $F'$ such that $S=S_{F'}$. By Lemma~\ref{kontribuutiolemma} the divisor $\sum_{\nu\in S}E_\nu$ is then contributing.
But this is a contradiction, because $G$ is critically contributing.
\end{proof}

\begin{example}Consider the situation of Example~\ref{esimerkki2}. We now have
$$\lambda(f_5-1,5)=\frac{7}{27}>\frac{1}{4}.$$
It therefore follows from Theorem~\ref{kriittinenkriteeri} that $E_1+E_2+E_3$ does not critically 
contribute $\xi$.

\end{example}

\end{document}